\documentclass[reqno]{amsart}
\usepackage{graphicx,color}
\usepackage{amssymb,amsfonts,amsthm}
\usepackage[all,arc]{xy}
\usepackage{enumerate}
\usepackage{mathrsfs}
\newtheorem{thm}{Theorem}[section]

\newtheorem{lemma}[thm]{Lemma}
\newtheorem{prop}[thm]{Proposition}

\theoremstyle{definition}

\theoremstyle{remark}

\numberwithin{equation}{section}

\title[Asymptotic Plateau problem]{The Asymptotic Plateau Problem for \(p\)-Convex Hypersurfaces in Hyperbolic Space}

\author{Shujun Shi}

\address{School of Mathematical Sciences, Harbin Normal University,
Harbin 150025, Heilongjiang Province, China}
\email{shjshi@hrbnu.edu.cn}

\author{Zhenan Sui}

\address{Institute for Advanced Study in Mathematics of HIT, Harbin Institute of Technology, Harbin, China}
\email{suizhenan@126.com}

\begin{document}

\begin{abstract}
We solve the asymptotic Plateau problem in hyperbolic space for the curvature given by the normalized geometric mean of the $p$-fold sums of the principal curvatures. For every $n\ge3$, $2\le p\le n-1$, and $\sigma\in(0,1)$, each bounded domain $\Omega\subset\mathbb{R}^n$ with smooth mean-convex boundary admits a unique complete admissible vertical graph with prescribed curvature $\sigma$ and asymptotic boundary $\partial\Omega\times\{0\}$. The result extends the known cases $(n,p)=(3,2)$ and $(4,3)$ to the stated range. The key curvature estimate is uniform with respect to the positive boundary height in the approximating Dirichlet problems. The proof establishes two lower bounds for the full third-derivative quadratic form under the linearized constraints: a sharp universal bound with coefficient $2/p$ and a stronger direction-dependent bound when the relevant principal curvature lies in a specified range.
\end{abstract}

\maketitle

\smallskip


%\tableofcontents

\vspace{4mm}

\section{Introduction}
\label{sec:main}

We use the upper half-space model $\mathbb{H}^{n+1}=\{(x,x_{n+1})\in\mathbb{R}^{n+1}:x_{n+1}>0\}$ with metric $x_{n+1}^{-2}\sum_{i=1}^{n+1}dx_i^2$. Given $\sigma\in(0,1)$ and a union $\Gamma$ of disjoint smooth, closed, embedded $(n-1)$-dimensional submanifolds of $\partial_\infty\mathbb{H}^{n+1}=\mathbb{R}^n\times\{0\}$, we seek a smooth complete hypersurface $\Sigma$ satisfying
\begin{equation} \label{eq1-1}
f \big( \kappa[ \Sigma ] \big) = \sigma
\end{equation}
and having asymptotic boundary $\partial_\infty\Sigma=\Gamma$. Here $\kappa[\Sigma]=(\kappa_1,\ldots,\kappa_n)$ is its hyperbolic principal curvature vector, and $f$ is a smooth symmetric function of $n$ variables.

The minimal, constant-mean-curvature, and constant-Gauss-curvature versions of this problem have been studied extensively, beginning with work by Anderson \cite{Anderson1982,Anderson1983},  Hardt--Lin \cite{HL1987},  Lin \cite{Lin1989}, Labourie \cite{Labourie1991}, Rosenberg--Spruck \cite{RS1994}, Nelli--Spruck \cite{NS1996}, Tonegawa \cite{Tonegawa1996}, Guan--Spruck \cite{GS00}. Let $\Omega$ be a smooth bounded domain (a nonempty connected open set). For a positive vertical graph of $u$ over $\Omega$, the asymptotic boundary condition $\Gamma=\partial\Omega\times\{0\}$ becomes $u=0$ on $\partial\Omega$. The corresponding fully nonlinear equation degenerates there.

Guan--Spruck--Szapiel \cite{GSS09} and Guan--Spruck \cite{GS10} developed an approximation theory for a broad class of symmetric, elliptic, concave, homogeneous curvature functions by studying the Dirichlet problem at positive boundary height
\begin{equation} \label{eq1-4}
\left\{ \begin{aligned}
f( \kappa [ u ] ) = & \sigma \quad \text{ in } & \Omega, \\
u = & \epsilon \quad \text{ on } & \partial\Omega,
\end{aligned} \right.
\end{equation}
Here $\kappa[u]$ denotes the principal curvature vector of the graph of $u$, and $\epsilon>0$ is fixed and sufficiently small. The curvature bounds needed for the limit $\epsilon\downarrow0$ must be uniform in $\epsilon$.

For mean-convex domains, Guan--Spruck proved solvability for every $\sigma\in(0,1)$ and every sufficiently small $\epsilon>0$. Their global curvature estimate is uniform in $\epsilon$ under the additional restriction $\sigma>\sigma_0\approx0.3703$. Xiao \cite{Xiao2013} lowered the threshold to $0.14596<\sigma_0<0.14597$ in the curvature flow setting. Guan, Spruck and Xiao \cite{GS11,GSX14} obtained existence over the full interval $\sigma\in(0,1)$ for locally strictly convex hypersurfaces, assuming that the curvature function is defined on the positive cone and vanishes on its boundary.

Fix integers $n\ge3$ and $2\le p\le n-1$, and write
\begin{equation*}
q=n-p,\qquad N=\binom np,\qquad
\mathcal{I}_p = \{I\subset\{1,\ldots,n\}:|I|=p\}.
\end{equation*}
Unless otherwise stated, lower-case indices range from $1$ to $n$, and
subset indices $I,J$ range over $\mathcal{I}_p$. Summation signs are displayed explicitly, including in tensor contractions;
repeated indices alone do not imply summation. We use $I_n$ for the
$n\times n$ identity matrix, reserving $I,J$ for $p$-element subsets.
For $\kappa\in\mathbb{R}^n$, set
\begin{equation*}
H = \sum_i \kappa_i, \qquad \lambda_I = \sum_{i\in I} \kappa_i, \qquad
\mathcal{P}_p= \{\kappa: \lambda_I > 0 \text{ for all } I \in \mathcal{I}_p \}.
\end{equation*}
On $\mathcal{P}_p$, define $f: \mathcal{P}_p \to (0, \infty)$ by
\begin{equation}
f (\kappa) = \frac{1}{p} \bigg( \prod_{I \in \mathcal{I}_p} \lambda_I \bigg)^{1/N}.
\label{eq:operator}
\end{equation}
A hypersurface is called $p$-convex if its principal curvature vector belongs to $\mathcal{P}_p$. Thus $p$-convexity here means that every $p$-sum is positive. We also call such a hypersurface admissible. A positive graphing function is admissible if its graph is admissible. For a self-adjoint shape operator $A$, write
\[ F( A ) = f \big( \lambda(A) \big), \]
where $\lambda(A)$ denotes the eigenvalue vector of the shape operator $A$. Thus equation \eqref{eq1-1} can be written equivalently as
\[ F(A) = f \big( \lambda(A) \big) = \sigma. \]

The cone $\mathcal{P}_p$ is related to the $p$-positivity cones studied in \cite{HarveyLawson2013}. For the same product operator in Euclidean space, Dong \cite{Dong2023,Dong2024} proved curvature and Dirichlet estimates. Dong's estimates for general dependence on the normal require $p \geq n/2$. For data independent of the normal, Dong \cite[Theorem~1.5]{Dong2024} treats $1\le p\le n$ under its stated strict-convexity, monotonicity, and subsolution hypotheses. Related star-shaped problems were studied in space forms \cite{Lu-Zhong} and warped product manifolds \cite{YangLuWarped2026}. Gong--Tu \cite{GongTu2026} studied closed star-shaped hypersurfaces in Euclidean space for elementary symmetric functions of the $p$-sums; their results include the product endpoint for certain prescribed functions. See also \cite{CDH2023,CJ2021,Mei2023} for related Hessian estimates. Our hyperbolic asymptotic boundary problem includes $p<n/2$; the required estimate is uniform in the approximating boundary height $\epsilon$.

The operator $f$ and its admissible cone should be distinguished from the elementary symmetric curvature functions
\[ \sigma_k (\kappa) = \sum_{1 \leq i_1 < \ldots < i_k \leq n} \kappa_{i_1} \cdots \kappa_{i_k} \]
and their cones
\[\Gamma_k = \big\{ \kappa \in \mathbb{R}^n \big\vert \sigma_j (\kappa) > 0, \, j = 1, \ldots, k \big\}. \]
For the elementary symmetric curvatures, Lu \cite{Lu2023} and Wang \cite{WangB-Adv,WangB-arXiv} obtained asymptotic Plateau results for particular indices. Mei--Yan \cite{MeiYan2026} treat the full intermediate range $2\le k\le n-1$; see also \cite[Remark~1.5 and Appendix~A]{WangHigher2026}. Adjacent quotients are treated in \cite{WangB-MRL}; the erratum in \cite{WangB-arXiv} corrects the scope of other quotient results stated in earlier versions of that work.

For the operator \eqref{eq:operator}, Chen--Sui--Sun \cite[Theorem~1.10]{CSS2025} proved existence for all $\sigma\in(0,1)$ when $(n,p)=(3,2)$. Sui \cite[Theorem~1.9]{Sui2026} established the corresponding result when $(n,p)=(4,3)$. Chen--Sui--Sun also proved a general-dimensional result for the curvature functions
\[
\frac{\sigma_k^{1/k} \big( H I_n - A \big)}{(n-1) \binom nk^{1/k}},\qquad 1\le k<n,
\]
in \cite[Theorem~1.7]{CSS2025}. The product endpoint $k = n$ is excluded from that theorem. Here $\sigma_k$ of a symmetric matrix means the $k$th elementary symmetric function of its eigenvalues. Sui's constrained minimization formula for this endpoint holds in arbitrary dimension \cite[Theorem~4.22]{Sui2026}. Sui obtained the corresponding uniform global curvature estimate in dimension four. We estimate the full third-order quadratic form under its two linear constraints and obtain quantitative lower bounds for every $2\le p\le n-1$. We then use these bounds in one maximum-principle argument for every $n\ge3$, $2\le p\le n-1$, and $\sigma\in(0,1)$.

For a vertical graph $\Sigma=\{(x,u(x)):x\in\Omega\}$, let $\mathbf{n}$ be the upward hyperbolic unit normal and let $\mathbf{D}$ denote the ambient hyperbolic connection. We define the shape operator by $A(X) = - \mathbf{D}_X \mathbf{n}$. With this convention, horizontal horospheres have principal curvatures $1$. The vertical component of the Euclidean unit normal is
\begin{equation}
\nu^{n + 1} = (1 + |Du|^2)^{-1/2}.
\label{eq:angle-def}
\end{equation}
We write $D$ for Euclidean differentiation and $\nabla$ for the Levi--Civita connection
of the induced metric on $\Sigma$.
The quantity $H=\operatorname{tr}A$ is the unnormalized mean curvature.
In graph coordinates $u_i=D_i u$, while in a local tangent frame
$u_i=\nabla_{\tau_i}u$; the frame will always be specified when the latter
notation is used. Likewise, $h_{ijk}$ denotes the covariant derivative
$\nabla_{\tau_k}h(\tau_i,\tau_j)$ in such a frame.
The norm $|\cdot|$ is Euclidean for vectors and tensors on the base domain and is induced by the hypersurface metric for tangent vectors and tensors. In particular,
$|D^2 u|^2 = \sum_{i,j}(D_{ij}u)^2$.

All function-space norms are Euclidean. For integers $k\ge0$ and
$0<\alpha\le1$, the notation
$C^{k,\alpha}(\overline\Omega)$ means that the derivatives up to order
$k$ extend continuously to the closure and that the derivatives of
order $k$ are H\"older continuous there. In particular,
$C^{1,1}(\overline\Omega)$ means that the extended gradient is
Lipschitz on $\overline\Omega$. Convergence in
$C^\infty_{\mathrm{loc}}(\Omega)$ means uniform convergence of every
finite-order derivative on each compact subset of $\Omega$.

We call $\partial\Omega$ mean-convex when its Euclidean mean curvature is nonnegative, with spheres bounding balls assigned positive mean curvature. Our main existence result is as follows.

\begin{thm}[Asymptotic Plateau problem]
\label{thm:plateau}
Let $n,p$ be integers with $n\ge3$ and $2\le p\le n-1$, and let $\Omega\subset\mathbb{R}^n$ be a bounded domain with smooth mean-convex boundary. For each $\sigma\in(0,1)$, there is a unique positive admissible solution of
\begin{equation}
F(A[u])=\sigma\quad\text{in }\Omega,\qquad u=0\quad\text{on }\partial\Omega,
\label{eq:asymptotic-problem}
\end{equation}
with $F$ defined by \eqref{eq:operator}.
This solution satisfies
\begin{gather*}
u\in C^\infty(\Omega)\cap C^1(\overline\Omega),\qquad
u^2\in C^{1,1}(\overline\Omega),
\\
\nu^{n + 1} \geq \sigma \quad\text{in }\Omega,\qquad
\nu^{n + 1} = \sigma \quad\text{on } \partial\Omega,
\\
\sup_\Omega\max_i|\kappa_i[u]|+\sup_\Omega u|D^2u|\le C(n,p,\sigma,\Omega).
\end{gather*}
Its graph is complete and properly embedded, with asymptotic boundary $\partial\Omega\times\{0\}$.
\end{thm}

To pass to the asymptotic limit in the Guan--Spruck approximation, we use the following global curvature estimate. The angle lower bound is an explicit hypothesis of this estimate; mean-convexity supplies it for the approximating solutions.
\begin{thm} \label{thm:main}
Let $n,p$ be integers with $n\ge3$ and $2\le p\le n-1$, and fix
$\sigma>0$. Let $\Omega\subset\mathbb R^n$ be a bounded smooth domain and
$u\in C^4(\Omega)\cap C^2(\overline\Omega)$ be positive on
$\overline\Omega$ and admissible. If its graph $\Sigma$ satisfies
\eqref{eq1-1} with $f$ defined by \eqref{eq:operator} and
$\nu^{n+1}\geq\sigma$ on $\Sigma$,
then
\begin{equation} \label{eq1-11}
\sup\limits_{\substack{X \in \Sigma \\ 1 \leq i \leq n}} \big| \kappa_i (X) \big| \leq C \bigg( 1 +   \max\limits_{\substack{X \in \partial\Sigma \\ 1 \leq i \leq n}} \big| \kappa_i (X) \big| \bigg),
\end{equation}
where $C=C(n,p,\sigma)$ is independent of $u$, $\Omega$, and the
positive boundary height. Here $\partial\Sigma$ is the finite boundary
of the positive graph, rather than its asymptotic boundary.
\end{thm}
The constant in Theorem~\ref{thm:main} is uniform in the positive boundary height; it may depend on $\sigma$ as $\sigma\downarrow0$.
The proof uses the test function $\frac{H}{(\nu^{n + 1})^\beta}$ and the constrained analysis of third-order terms developed in \cite[Sections~3--4]{Sui2026}. At a maximum of the test function, the derivative of $H$ and the linearized curvature equation give two constraints on the third derivatives. We use both to obtain quantitative lower bounds for the associated quadratic form. Let $S_i = \sum_{I \ni i} \lambda_I^{-1}$. For each admissible $\kappa$, each $i$, and every $t=(t_1,\ldots,t_n)\in\mathbb R^n$ satisfying the constraint below, the full quadratic form $Q_i(\kappa,t)$ of \eqref{eq: Qi}, including the off-diagonal spectral terms, satisfies
\[
Q_i \geq \frac{S_i}{H}
\max\bigg\{\frac2p,\frac{2H}{H+(n-p)|\kappa_i|}\bigg\}
\bigg( \sum_j t_j \bigg)^2,
\qquad \sum_j S_j t_j=0.
\]
Proposition~\ref{prop:universal} shows that $2/p$ is sharp over the admissible cone and that equality at any fixed admissible curvature vector requires $t=0$. The direction-dependent coefficient is at least $3/2$ when $|\kappa_i|\le H/[3(n-p)]$. Related uses of third-order concavity appear in \cite{GRW2015,RW2019}.

At an interior maximum of $H/(\nu^{n+1})^\beta$ with $\beta=1+1/p$, the universal bound controls the negative principal curvatures. The directional bound controls the remaining terms in the maximum-principle inequality. The resulting curvature estimate, together with the Guan--Spruck approximation theory, yields the existence theorem.

The rest of this paper is organized as follows.
Section~\ref{sec:preliminary} records the graph equation, the structural properties of the operator, and geometric identities. Sections~\ref{sec:algebra}--\ref{sec:directional} establish the constrained quadratic estimates. Section~\ref{sec:first-absorption} proves the uniform global curvature estimate, finishing the proof of Theorems~\ref{thm:main} and Theorem~\ref{thm:plateau}.

\vspace{4mm}

\section{Preliminaries}
\label{sec:preliminary}

In this section, we collect some known facts.
The symbol $\delta_{ij}$ is the Kronecker delta,
and a superscript $T$ denotes transpose. Symmetric matrix inequalities
are understood in the sense of quadratic forms; positivity means
positive definiteness.

\vspace{2mm}

\subsection{Graph equation}
\label{sec:prelim}

For the vertical graph
\[ \Sigma = \Big\{ \big( x, u(x) \big) \big| x \in \Omega \Big\},  \]
where $\Omega$ is a smooth bounded domain in $\mathbb{R}^n \equiv \partial_{\infty} \mathbb{H}^{n + 1}$,
set $w=\sqrt{1+|Du|^2}$ and $\gamma=(\gamma^{ij})$, where
\[
\gamma^{ij}=\delta_{ij}-\frac{u_i u_j}{w(1+w)}.
\]
The induced metric and a symmetric matrix representing the shape operator are
\begin{equation}
g_{ij}=u^{-2}(\delta_{ij}+u_i u_j),\qquad
A[u] = \frac{1}{w} \Big( I_n + u \gamma D^2 u \gamma \Big).
\label{eq:graph-matrix}
\end{equation}
The eigenvalues of $A[u]$ are the hyperbolic principal curvatures, with the sign convention specified above.

The hyperbolic and Euclidean geometric quantities are related as follows.
\[ \mathbf{n} = u \nu, \]
where
\[\nu = \frac{( - D u, 1 )}{w}   \]
is the upward Euclidean unit normal vector field on $\Sigma$.
\begin{equation} \label{eq2-11}
h_{ij} = \frac{1}{u} \tilde{h}_{ij} + \frac{\nu^{n+1}}{u^2} \tilde{g}_{ij}
\end{equation}
and
\begin{equation} \label{eq2-12}
\kappa_i =  u \tilde{\kappa_i} + \nu^{n+1},  \quad i = 1, \ldots, n,
\end{equation}
where $\tilde g$ and $\tilde h$ are the Euclidean first and second fundamental forms, respectively, and $\tilde{\kappa} = (\tilde{\kappa}_1, \cdots, \tilde{\kappa}_n)$  are its Euclidean principal curvatures. Untilded geometric quantities refer to the hyperbolic induced metric and second fundamental form; see \cite{GSS09,GS10,GS11,GSX14} for these formulas.

\vspace{2mm}

\subsection{Structural properties}
\label{sec:structural}

Define
\[
G(r,\xi,z) = F \Bigg( \frac{1}{\sqrt{1+|\xi|^2}}
\Big( I_n + z \gamma(\xi) r \gamma(\xi) \Big) \Bigg).
\]
Here $r$ is a symmetric $n\times n$ matrix, $\xi\in\mathbb{R}^n$, $z>0$, and $\gamma(\xi)$ is obtained by replacing $Du$ by $\xi$ in the definition of $\gamma$. Since $F$ is elliptic and concave on its admissible matrix cone, $G$ is elliptic and concave in $r$ whenever $z>0$ and the matrix is admissible.

The exterior-algebra characterization of $p$-positivity, also used in \cite{GongTu2026}, expresses $F$ as a root of a determinant. For a symmetric matrix $A$, let $A^{[p]}$ denote the additive action
\[
A^{[p]}(v_1\wedge\cdots\wedge v_p)
=\sum_{j=1}^p v_1\wedge\cdots\wedge Av_j\wedge\cdots\wedge v_p.
\]
Its eigenvalues are the $p$-sums $\lambda_I$, and hence
\begin{equation}
F(A)=p^{-1}\det(A^{[p]})^{1/N},\qquad
\lambda(A)\in\mathcal{P}_p \Longleftrightarrow A^{[p]}>0.
\label{eq:additive-action}
\end{equation}
The admissible matrix cone is convex and is preserved by adding positive semidefinite matrices. Formula \eqref{eq:additive-action} shows that $F$ is smooth even at matrices with repeated eigenvalues. For general results on spectral differentiation, see \cite{Ball1984}.

The matrix function $F$ is defined on
\[
\mathcal K_p
=\{A\in \mbox{Sym}(n):A^{[p]}>0\}.
\]
Here $\mbox{Sym}(n)$ is the space of real symmetric
$n\times n$ matrices. Accordingly, the domain of the graph operator is
\[
\mathcal U_p
=\bigg\{(r,\xi,z):r\in\mbox{Sym}(n), \xi\in\mathbb{R}^n, z>0,
 \frac{I_n + z \gamma(\xi) r \gamma(\xi)}
{\sqrt{1+|\xi|^2}} \in\mathcal K_p \bigg\}.
\]

We write $\mathbf{1}=(1,\ldots,1)^T$ and $f_i=\partial f/\partial\kappa_i$. Also, we write
\[ F^{ij}(A) = \frac{\partial F}{\partial a_{ij}} (A), \quad F^{ij, kl} (A) = \frac{\partial^2 F}{\partial a_{ij} \partial a_{kl}} (A) . \]

The following lemma verifies that all structural conditions in Guan--Spruck \cite{GS10} are satisfied.
\begin{lemma}[Structural properties]
\label{lem:structure}
The cone $\mathcal{P}_p$ is open, symmetric and convex, contains the positive orthant, and is contained in $\{H>0\}$. The function $f$ is smooth, symmetric, positive, elliptic, concave and homogeneous of degree one on $\mathcal{P}_p$. It extends continuously to $\overline{\mathcal{P}_p}$, vanishes on $\partial\mathcal{P}_p$, and satisfies $f(\mathbf{1})=1$. Moreover,
\begin{equation}
f(\kappa) \leq H/n,\qquad \sum_i f_i(\kappa)\ge1.
\label{eq:structural-trace}
\end{equation}
Let $e_n$ denote the last coordinate unit vector. Then, uniformly for $\kappa$ in a sufficiently small fixed neighborhood of $\mathbf{1}$,
\begin{equation}
f(\kappa+Re_n)\longrightarrow\infty\quad\text{as }R\longrightarrow\infty.
\label{eq:ray-growth}
\end{equation}
\end{lemma}
\begin{proof}
The stated properties of $\mathcal{P}_p$ follow from its definition and the identity
$\sum_I\lambda_I=\binom{n-1}{p-1}H$. The geometric mean is concave on the positive orthant.
Composing the geometric mean with the linear map $\kappa\mapsto(\lambda_I)_I$ proves that $f$ is concave. At the matrix level, \eqref{eq:additive-action} gives the same conclusion: $A\mapsto A^{[p]}$ is linear and $B\mapsto\det(B)^{1/N}$ is concave on positive definite $N\times N$ matrices. Differentiation gives
$f_i=(f/N)\sum_{I\ni i}\lambda_I^{-1}>0$.
The remaining basic properties follow directly from the product formula. Concavity at $\mathbf{1}$, symmetry and homogeneity imply $f\le H/n$; concavity at $\kappa$, applied to $\mathbf{1}$, gives
$1\le f(\kappa)+\sum_i f_i(\kappa)(1-\kappa_i)=\sum_i f_i(\kappa)$.
Finally, exactly $\binom{n-1}{p-1}$ factors contain the index $n$. Near $\mathbf{1}$, all other factors are uniformly bounded away from zero and the factors indexed by sets containing $n$ grow linearly in $R$. Thus $f(\kappa+Re_n)\ge cR^{p/n}$ for $R\ge1$, proving \eqref{eq:ray-growth}.
\end{proof}

\vspace{2mm}

\subsection{Geometric identities on hypersurface $\Sigma$}

Let $\tilde{\nabla}$ be the Levi-Civita connection induced from the Euclidean ambient space, and the other notations are as mentioned before. We have the following identities, whose proof can be found in \cite{GSS09, GS10, GS11, GSX14, Sui2019}.

\begin{lemma}  \label{Lemma1}
In any local frame on $\Sigma \subset \mathbb{R}^{n+1}$, the following identities hold.
\begin{equation} \label{eq2-2}
\tilde{g}^{kl} u_k u_l  = |\tilde\nabla u|^2 = 1 - (\nu^{n + 1})^2,
\end{equation}
\begin{equation}  \label{eq2-3}
\tilde{\nabla}_{ij} u = \tilde{h}_{ij} \nu^{n+1} \quad \mbox{and} \quad \tilde{\nabla}_{ij} x_{k} = \tilde{h}_{ij} \nu^{k}, \quad k = 1, \ldots, n,
\end{equation}
\begin{equation}  \label{eq2-4}
(\nu^{n+1})_i = - \tilde{h}_{ij} \tilde{g}^{j k} u_k,
\end{equation}
\begin{equation} \label{eq2-5}
\tilde{\nabla}_{ij} \nu^{n+1} = - \tilde{g}^{kl} ( \nu^{n+1} \tilde{h}_{il} \tilde{h}_{kj} + u_l \tilde{\nabla}_k \tilde{h}_{ij} ) .
\end{equation}
\end{lemma}

\begin{lemma}  \label{Lemma6}
Let $\Sigma$ be a smooth hypersurface in $\mathbb{H}^{n+1}$ satisfying \eqref{eq1-1}. In a local orthonormal frame on $\Sigma$, we have the identity
\begin{equation}  \label{eq4-1}
\begin{aligned}
F^{ij} \nabla_{ij} \nu^{n+1}
= &  \Big( 1 + (\nu^{n+1})^2 \Big) \sigma - \nu^{n+1} \Big( \sum f_i +  \sum f_i \kappa_i^2 \Big) \\ & + \frac{2}{u^2} F^{ij} u_i u_j \big( \nu^{n + 1} - \kappa_j \big).
\end{aligned}
\end{equation}
\end{lemma}

\vspace{2mm}

\subsection{Cone identities}

Throughout the following sections, the integers $n \geq 3$ and $2 \leq p \leq n - 1$ are fixed, and $\kappa \in \mathcal{P}_p$. For convenience, we define
\begin{equation} \label{eq:weights}
\begin{aligned}
S = \sum_{I \in \mathcal{I}_p} \frac{1}{\lambda_I}, \qquad
S_i = \sum_{I \ni i} \frac{1}{\lambda_I}.
\end{aligned}
\end{equation}

\begin{lemma}[Elementary identities]
\label{lem:elementary}
On $\mathcal{P}_p$ one has
\begin{gather}
H > 0,\qquad - \frac{p-1}{q} H < \kappa_i < H,
\label{eq:cone-bounds}\\
0 < \frac{S_i}{S} < 1, \qquad  \sum_i \frac{S_i}{S} = p,\qquad
\sum_i S_i \kappa_i = N,
\label{eq:weight-sums}\\
S_i - S_j = (\kappa_j - \kappa_i) \sum_{\substack{K \subset \{1, \ldots, n\} \setminus \{i,j\} \\ |K| = p - 1}}
\frac{1}{\lambda_{K \cup \{i\}} \lambda_{K \cup \{j\}}}  \qquad  i \neq j .
\label{eq:a-difference}
\end{gather}
Moreover,
\[ f_i := \frac{\partial f}{\partial \kappa_i} = \frac{f}{N} S_i > 0, \qquad \sum f_i = \frac{p f}{N} S.  \]
If $\kappa_1 \ge \cdots \ge \kappa_n$, then $S_1 \le \cdots \le S_n$ and $\frac{S_n}{S} \ge \frac{p}{n}$.
\end{lemma}

\begin{proof}
Summing all positive $p$-sums gives
\[
\sum_{I \in \mathcal{I}_p} \lambda_I = \binom{n-1}{p-1} H > 0. \label{eq: S}
\]
The average of the $p$-sums containing $i$ equals
\[
\kappa_i + \frac{p-1}{n-1}(H - \kappa_i) > 0,
\]
which gives the lower bound in \eqref{eq:cone-bounds}. The average of the $p$-sums not containing $i$ is $p (H - \kappa_i)/(n-1) > 0$, giving the upper bound in the same formula.

Also
\[ \sum_i S_i = \sum_i  \sum_{I \ni i} \frac{1}{\lambda_I} =  \sum_{I \in \mathcal{I}_p} \frac{1}{\lambda_I} \sum_{i \in  I} 1 = p S, \]
and
\[
\sum_i S_i \kappa_i =  \sum_i  \sum_{I \ni i} \frac{\kappa_i}{\lambda_I}
= \sum_{I \in \mathcal{I}_p} \frac{\sum_{i \in I} \kappa_i}{\lambda_I} = N.
\]
To prove \eqref{eq:a-difference}, we note that
\[ \begin{aligned}
S_i - S_j = & \sum_{I \ni i} \frac{1}{\lambda_I} -  \sum_{I \ni j} \frac{1}{\lambda_I} \\
= & \sum_{\substack{K \subset \{1, \ldots, n\} \setminus \{i,j\} \\ |K| = p - 1}} \Bigg( \frac{1}{\lambda_{K \cup \{i\}}} - \frac{1}{\lambda_{K \cup \{j\}}} \Bigg) \\
= & \sum_{\substack{K \subset \{1, \ldots, n\} \setminus \{i,j\} \\ |K| = p - 1}} \frac{\kappa_j - \kappa_i}{\lambda_{K \cup\{i\}} \lambda_{K \cup \{j\}}}.
\end{aligned} \]
The derivative formula follows by differentiating $\log f$. The ordering and the bound for the largest weight follow from \eqref{eq:a-difference} and $\sum_i \frac{S_i}{S} = p$.
\end{proof}

\vspace{4mm}

\section{Setup for uniform global curvature estimate}
\label{sec:algebra}

We work under the hypotheses of Theorem~\ref{thm:main}, including the assumed angle bound $\nu^{n+1}\geq\sigma$. For the approximating Dirichlet solutions, mean-convexity of $\partial\Omega$ gives this bound by \cite{GS10}.
We choose the test function
\begin{equation*}
\frac{H}{({\nu}^{n+1})^{\beta}}.
\end{equation*}
Here $\beta>1$ will be chosen later. Suppose the test function attains its maximum at an interior point $X_0$; if its maximum is attained on $\partial\Sigma$, the boundary term in \eqref{eq1-11} gives the desired bound.
Let $\tau_1, \ldots, \tau_n$ be a smooth local orthonormal frame field about
$X_0$ such that $h_{ij}(X_0) = \kappa_i \delta_{ij}$, where
$\kappa_1 \geq \cdots \geq \kappa_n$ are the hyperbolic principal curvatures of
$\Sigma$ at $X_0$.
Write $H_i:=\nabla_{\tau_i}H=\sum_j h_{jji}$.
At $X_0$, the function $\log H-\beta\log\nu^{n+1}$ has a local maximum. Its first and second derivatives therefore satisfy
\begin{equation} \label{eq4-2}
\frac{\sum_j h_{jji}}{H} - \frac{\beta \nabla_i \nu^{n + 1}}{\nu^{n + 1}} = 0,
\end{equation}
and
\begin{equation} \label{eq4-3}
\frac{\sum_j h_{jjii}}{H} - \frac{\beta \nabla_{ii} \nu^{n + 1}}{\nu^{n + 1}} - \frac{\beta (\beta - 1)( \nabla_i \nu^{n + 1} )^2}{(\nu^{n + 1})^2} \leq 0.
\end{equation}
Unless otherwise stated, all quantities in the following calculations are evaluated at $X_0$.

Differentiating equation \eqref{eq1-1} twice yields
\begin{equation} \label{eq4-23}
\sum_i F^{ii} h_{iij} = 0,
\end{equation}
and
\begin{equation}  \label{eq4-4}
\sum_i F^{ii} h_{iijj} + \sum_{k l r s} F^{k l, rs} h_{klj} h_{rsj} = 0 .
\end{equation}
Moreover, by Gauss equation we know that
\begin{equation} \label{eq2G-4}
h_{iijj} = h_{jjii} + ( \kappa_i \kappa_j - 1 )( \kappa_i -
\kappa_j ).
\end{equation}

By \eqref{eq2-4} and \eqref{eq2-11}, we have
\begin{equation} \label{eq4-24}
(\nu^{n + 1})_i = \frac{u_i}{u} (\nu^{n + 1} - \kappa_i).
\end{equation}
Substituting \eqref{eq4-24} into \eqref{eq4-2} yields
\begin{equation} \label{eq4-6}
\sum_j h_{jji} = \frac{\beta}{\nu^{n + 1}} \frac{u_i}{u} (\nu^{n+1} - \kappa_i) H.
\end{equation}

Combining \eqref{eq4-3}, \eqref{eq4-4}, \eqref{eq2G-4}, \eqref{eq4-1}  and  \eqref{eq4-24}, and for $\kappa_1$ sufficiently large, we have
\begin{equation} \label{eq4-8}
\begin{aligned}
& (\beta - 1) H \bigg( \sum f_i + \sum f_i \kappa_i^2 \bigg)  - \sum_{klrsj} F^{kl, rs} h_{klj} h_{rsj}
\\ &
+ \frac{2 \beta}{\nu^{n + 1}} H \sum f_i \frac{u_i^2}{u^2} \big( \kappa_i - \nu^{n+1} \big) - \frac{\beta (\beta - 1)}{(\nu^{n + 1})^2} H \sum f_i  \frac{u_i^2}{u^2} (\nu^{n + 1} - \kappa_i)^2 < 0.
\end{aligned}
\end{equation}

\vspace{2mm}

\subsection{The constrained quadratic form}

Recall that
\begin{equation} \label{eq3-1}
F( A ) = \frac{1}{p} \mbox{det}^{\frac{1}{N}} \Big( A^{[p]} \Big) =  \frac{1}{p} \bigg( \prod_{I \in \mathcal{I}_p} \lambda_I \bigg)^{\frac{1}{N}}.
\end{equation}
Then we have
\begin{equation} \label{eq: ln-equation}
\ln F( A ) = {\frac{1}{N}} \sum_{I \in \mathcal{I}_p} \ln \lambda_I - \ln p.
\end{equation}
Taking derivatives of \eqref{eq: ln-equation} yields,
\[ \frac{F^{kl}}{F} =  {\frac{1}{N}} \sum_{I \in \mathcal{I}_p}  \lambda_I^{-1} \frac{\partial \lambda_I}{\partial a_{kl}} , \]
and
\[ \frac{F^{kl, rs}}{F} - \frac{F^{kl} F^{rs}}{F^2} =  {\frac{1}{N}} \sum_{I \in \mathcal{I}_p}  \lambda_I^{-1} \frac{\partial^2 \lambda_I}{\partial a_{kl} \partial a_{rs}} - {\frac{1}{N}} \sum_{I \in \mathcal{I}_p} \lambda_I^{-2} \frac{\partial \lambda_I}{\partial a_{kl}} \frac{\partial \lambda_I}{\partial a_{rs}}. \]
Together with \eqref{eq4-23}, the formulas in \cite{Ball1984} and Codazzi equation we know that at $X_0$,
\begin{equation} \label{eq: concavity}
\begin{aligned}
- \sum_{klrsj} F^{kl, rs} h_{klj} h_{rsj}
= & - \frac{F}{N} \sum_{klrsj} \sum_{i} \frac{\partial^2 \kappa_i}{\partial a_{kl} \partial a_{rs}}  h_{klj} h_{rsj} \sum_{I \ni i} \lambda_I^{-1} \\
+ \frac{F}{N} \sum_{klrsj} \sum_{I \in \mathcal{I}_p} & \lambda_I^{-2} \frac{\partial \sum_{i \in I} \kappa_i}{\partial a_{kl}} \frac{\partial \sum_{m \in I} \kappa_m}{\partial a_{rs}}  h_{klj} h_{rsj} \\
= \sum_{ij} \sum_{l \neq i} \frac{f_i - f_l}{\kappa_l - \kappa_i}  h_{ijl}^2 & + \frac{F}{N} \sum_{j} \sum_{I \in \mathcal{I}_p} \lambda_I^{-2} \bigg( \sum_{i \in I} h_{iij} \bigg)^2 \\
\geq  2 \sum_{i \neq j} \frac{f_i - f_j}{\kappa_j - \kappa_i}  h_{jji}^2 + & \frac{F}{N} \sum_{i} \sum_{I \in \mathcal{I}_p} \lambda_I^{-2} \bigg( \sum_{j \in I} h_{jji} \bigg)^2.
\end{aligned}
\end{equation}

For each fixed $i$,  set
\[ s^{(i)} = s :=  \frac{\beta}{\nu^{n + 1}} \frac{u_i}{u} (\nu^{n+1} - \kappa_i) H,  \]
\[ t^{(i)}_j = t_j := h_{jji}, \qquad j = 1, \ldots, n, \]
and
\[ t^{(i)} = t := (t_1, \ldots, t_n)^T .  \]
We shall omit the superscript $(i)$ whenever there is no ambiguity.
Also let
\begin{equation} \label{eq: Qi}
\begin{aligned}
Q_i := & 2 \sum_{j \neq i} \frac{S_i - S_j}{\kappa_j - \kappa_i}  t_j^2 +  \sum_{I \in \mathcal{I}_p} \lambda_I^{-2} \bigg( \sum_{j \in I} t_j \bigg)^2 \\
= & 2 \sum_{j \neq i} \sum_{\substack{K \subset \{1, \ldots, n\} \setminus \{i,j\} \\ |K| = p - 1}}
\frac{1}{\lambda_{K \cup \{i\}} \lambda_{K \cup \{j\}}} t_j^2 +  \sum_{I \in \mathcal{I}_p} \lambda_I^{-2} \bigg( \sum_{j \in I} t_j \bigg)^2.
\end{aligned}
\end{equation}
Then \eqref{eq: concavity} becomes
\begin{equation} \label{eq: Concavity}
- \sum_{klrsj} F^{kl, rs} h_{klj} h_{rsj} \geq  \frac{F}{N} \sum_i Q_i.
\end{equation}
As a result, \eqref{eq4-8} reduces to
\begin{equation} \label{eq: lead-inequality}
\begin{aligned}
& (\beta - 1) \bigg( \sum S_i + \sum S_i \kappa_i^2 \bigg)  +  \frac{\sum_i Q_i}{H}
\\ &
+ \frac{2 \beta}{\nu^{n + 1}} \sum S_i \frac{u_i^2}{u^2} \big( \kappa_i - \nu^{n+1} \big) - \frac{\beta (\beta - 1)}{(\nu^{n + 1})^2} \sum S_i  \frac{u_i^2}{u^2} (\nu^{n + 1} - \kappa_i)^2 < 0.
\end{aligned}
\end{equation}

\vspace{2mm}

\subsection{Lagrange multiplier method}

For each fixed $i$, we find the minimum of $Q_i$ under the constraint $\sum_j t_j = s$ and $\sum_j S_j t_j = 0$.

\begin{prop}[Exact constrained minimization]
\label{prop:schur}
Fix $\kappa\in\mathcal{P}_p$ and $i \in \{1, \ldots, n\}$. For $I \in \mathcal{I}_p$, let $\mathbf{1}_I \in \mathbb{R}^n$ be its indicator vector. Let $\mathcal{B} \in \mathbb{R}^{N \times n}$ be the $p$-subset incidence matrix, whose row indexed by $I$ is $\mathbf{1}_I^T$, and set
\begin{gather*}
M_i = \mathcal{B}^T \mbox{diag}\Big( \lambda_I^{-2} \Big) \mathcal{B}
+ 2 \mbox{diag}(d_1, \ldots, d_n), \\
d_i = 0, \qquad
d_j = \sum_{\substack{K \subset \{1, \ldots, n\} \setminus \{i,j\} \\ |K| = p - 1}}
\frac{1}{\lambda_{K \cup \{i\}} \lambda_{K \cup \{j\}}}   \qquad  j \neq i.
\end{gather*}
Then $M_i \in \mbox{Sym}(n)$ is positive definite and $Q_i = t^T M_i t$. If the components of $\kappa$ are not all equal, put
\[
\mathsf U = \left(
              \begin{array}{ccc}
                1 & \cdots & 1 \\
                S_1 & \cdots & S_n \\
              \end{array}
            \right)^T, \quad
            c = \binom{1}{0}, \quad
\mathsf G_i = \mathsf U^T M_i^{-1} \mathsf U.
\]
The matrix $\mathsf G_i$ is positive definite, and
\begin{equation}
\mathcal{Q}_i := \min_{\substack{\sum_j t_j = 1, \\ \sum_j S_j t_j = 0}} Q_i
= c^T \mathsf G_i^{-1} c, \qquad
t_*(s) =  s M_i^{-1} \mathsf U \mathsf G_i^{-1} c.
\label{eq:schur}
\end{equation}
For every $t$ satisfying $\sum_j t_j = s$ and $\sum_j S_j t_j = 0$,
\begin{equation}
Q_i = \mathcal{Q}_i s^2 + \big( t - t_*(s) \big)^T M_i \big( t - t_*(s) \big).
\label{eq:schur-remainder}
\end{equation}
\end{prop}

\begin{proof}
Applying Lagrange multiplier method, we obtain the constrained minimum and the quadratic remainder; compare \cite[Theorem~4.22]{Sui2026}.

If $\mathcal{B} t = 0$, comparison of two $p$-sets differing only by the replacement of $i$ with $j$ gives $t_i = t_j$. Such pairs of sets exist because $p \leq n-1$. Hence all components of $t$ are equal. Since their sum over any $p$-set is zero, every component vanishes. Thus $\mathcal{B}$ has full column rank and $M_i > 0$.

By \eqref{eq:a-difference}, all $S_i$ coincide if and only if all $\kappa_i$ coincide. Consequently the columns of $\mathsf U$ are independent when $\kappa\notin\mathbb{R} \mathbf{1}$, and $\mathsf G_i > 0$. The displayed minimizer satisfies $\mathsf U^T t_*(s) = s c$ and $M_i t_*(s) \in \mbox{range} \mathsf U$. Therefore the cross term vanishes when $t - t_*(s)$ lies in $\ker \mathsf U^T$, proving \eqref{eq:schur-remainder} and \eqref{eq:schur}.
\end{proof}

When $\kappa \in \mathbb{R} \mathbf{1}$, $\mathsf G_i^{-1}$ and $\mathcal{Q}_i$ are undefined. The coercivity estimates below are stated directly for $Q_i$ on the constraint space and remain valid at such vectors.

Proposition \ref{prop:schur} proves that the minimum is attained, a fact used to establish strictness in Proposition \ref{prop:universal}. Sections~\ref{sec:universal} and~\ref{sec:directional} prove the quantitative bounds needed for uniform global curvature estimate directly, without evaluating the inverse matrices in \eqref{eq:schur}.

\vspace{4mm}

\section{A sharp universal lower bound}
\label{sec:universal}

\begin{lemma}[An estimate for $\kappa_i\leq 0$]
\label{lem:negative}
Let $\kappa\in\mathcal{P}_p$ and $t \in \mathbb{R}^n$. If $\kappa_i \leq 0$ and $\sum_j S_j t_j = 0$, then
\begin{equation}
Q_i \geq \frac{2 S_i}{H - q \kappa_i} \bigg( \sum_j t_j \bigg)^2
\geq \frac{2 S_i}{p H} \bigg( \sum_j t_j \bigg)^2.
\label{eq:negative-bound}
\end{equation}
\end{lemma}

\begin{proof}
For any $\kappa \in \mathcal{P}_p$ and $\sum S_j t_j = 0$, \eqref{eq:a-difference} gives
\begin{equation}
\begin{aligned}
S_i \sum_j t_j  = & S_i \sum_j t_j - \sum_j S_j t_j
= \sum_{j \neq i} (S_i - S_j) t_j \\
= & \sum_{j \neq i} \sum_{\substack{K \subset \{1, \ldots, n\} \setminus \{i,j\} \\ |K| = p - 1}}
\frac{1}{\lambda_{K \cup \{i\}} \lambda_{K \cup \{j\}}}  (\kappa_j - \kappa_i) t_j.
\label{eq:constraint-identity}
\end{aligned}
\end{equation}
 Using \eqref{eq:weight-sums} and $S = S_i + \sum_{J \not\ni i} \lambda_J^{-1}$, we obtain
\begin{equation}
\begin{aligned}
& \sum_{j \neq i} \sum_{\substack{K \subset \{1, \ldots, n\} \setminus \{i,j\} \\ |K| = p - 1}}
\frac{1}{\lambda_{K \cup \{i\}} \lambda_{K \cup \{j\}}}  (\kappa_j- \kappa_i)^2 \\
= & \sum_j (S_i - S_j)(\kappa_j - \kappa_i) \\
= & S_i (H - n \kappa_i) - N + p \kappa_i S \\
= & S_i (H - q \kappa_i) - N + p \kappa_i \sum_{J \not\ni i} \lambda_J^{-1}
\leq S_i (H - q \kappa_i).
\label{eq:Si}
\end{aligned}
\end{equation}
Consequently
\[ \begin{aligned}
& S_i^2 \bigg( \sum_j t_j \bigg)^2 = \Bigg( \sum_{j \neq i} \sum_{\substack{K \subset \{1, \ldots, n\} \setminus \{i,j\} \\ |K| = p - 1}}
\frac{1}{\lambda_{K \cup \{i\}} \lambda_{K \cup \{j\}}}  (\kappa_j - \kappa_i) t_j \Bigg)^2 \\
\leq & \sum_{j \neq i} \sum_{\substack{K \subset \{1, \ldots, n\} \setminus \{i,j\} \\ |K| = p - 1}}
\frac{1}{\lambda_{K \cup \{i\}} \lambda_{K \cup \{j\}}}  (\kappa_j - \kappa_i)^2  \sum_{j \neq i} \sum_{\substack{K \subset \{1, \ldots, n\} \setminus \{i,j\} \\ |K| = p - 1}}
\frac{1}{\lambda_{K \cup \{i\}} \lambda_{K \cup \{j\}}}  t_j^2 \\
\leq & \frac{1}{2} S_i ( H - q \kappa_i) Q_i.
\end{aligned} \]
This proves the first inequality in \eqref{eq:negative-bound}. The second follows from $H - q \kappa_i \leq p H$, which is a consequence of \eqref{eq:cone-bounds}.
\end{proof}

\begin{lemma}[An unconstrained estimate for $\kappa_i \geq 0$]
\label{lem:positive-universal}
Let $\kappa \in \mathcal{P}_p$. If $\kappa_i \geq 0$, then for every $t \in \mathbb{R}^n$,
\begin{equation}
Q_i \geq \frac{2 S_i}{p H} \bigg( \sum_j t_j \bigg)^2.
\label{eq:positive-universal}
\end{equation}
\end{lemma}

\begin{proof}
For $I \ni i$ and $j \notin I$, write
\[
I_j = (I\setminus\{i\})\cup\{j\}, \qquad
d_I = \lambda_I + \frac{1}{2} \sum_{j \notin I} \lambda_{I_j} > 0.
\]
Since $\sum_j t_j = \sum_{j \in I} t_j + \sum_{j \notin I} t_j$, Cauchy--Schwarz gives
\begin{equation}
\frac{\big( \sum_{j \in I} t_j \big)^2}{\lambda_I^2}
+ 2 \sum_{j \notin I} \frac{t_j^2}{\lambda_I \lambda_{I_j}}
\geq \frac{\big( \sum_j t_j \big)^2}{\lambda_I d_I}.
\label{eq:positive-local-cs}
\end{equation}
Also,
\begin{equation}
\sum_{I \ni i} \sum_{j \notin I}
\frac{t_j^2}{\lambda_I \lambda_{I_j}}
= \sum_{j \neq i} \sum_{\substack{K \subset \{1, \ldots, n\} \setminus \{i,j\} \\ |K| = p - 1}}
\frac{1}{\lambda_{K \cup \{i\}} \lambda_{K \cup \{j\}}} t_j^2.
\label{eq:positive-counting}
\end{equation}
Thus
\[ \begin{aligned}
Q_i \geq & Q_i - \sum_{J \not\ni i} \frac{\big( \sum_{j \in J} t_j \big)^2}{\lambda_J^2} \\
= &  2 \sum_{j \neq i} \sum_{\substack{K \subset \{1, \ldots, n\} \setminus \{i,j\} \\ |K| = p - 1}}
\frac{1}{\lambda_{K \cup \{i\}} \lambda_{K \cup \{j\}}} t_j^2 + \sum_{I \ni i} \lambda_I^{-2} \bigg( \sum_{j \in I} t_j \bigg)^2 \\
= &  2 \sum_{I \ni i} \sum_{j \notin I}
\frac{t_j^2}{\lambda_I \lambda_{I_j}} + \sum_{I \ni i} \lambda_I^{-2} \bigg( \sum_{j \in I} t_j \bigg)^2 \\
\geq &  \sum_{I \ni i} \frac{\big( \sum_j t_j \big)^2}{\lambda_I d_I}.
\end{aligned} \]

Applying Cauchy--Schwarz once more, we obtain
\begin{equation}
\begin{aligned}
Q_i \geq
\frac{S_i^2 \big( \sum_j t_j \big)^2}{\sum_{I \ni i} d_I / \lambda_I}
= \frac{2 S_i \big( \sum_j t_j \big)^2}{H - q \kappa_i + (q+1) \binom{n-1}{p-1}/S_i}.
\label{eq:positive-denominator}
\end{aligned}
\end{equation}
The equality follows from the identity
\[ \begin{aligned}
\sum_{I \ni i} \frac{d_I}{\lambda_I} = & \sum_{I \ni i} \bigg( 1 +  \frac{\sum_{j \notin I} \lambda_{I_j}}{2 \lambda_I} \bigg) =  \sum_{I \ni i} \bigg( 1 +  \frac{\sum_{j \notin I} (\lambda_I - \kappa_i + \kappa_j)}{2 \lambda_I} \bigg) \\
= & \sum_{I \ni i} \bigg( 1 +  \frac{q}{2} - q \frac{\kappa_i}{2 \lambda_I} + \frac{H - \lambda_I}{2 \lambda_I} \bigg) = \frac{1}{2} (q + 1) \binom{n-1}{p-1} + \frac{S_i}{2} (H - q \kappa_i).
\end{aligned} \]
The denominator in \eqref{eq:positive-denominator} is positive because all $d_I$ are positive.

The harmonic--arithmetic mean inequality gives
\[
\frac{\binom{n-1}{p-1}}{S_i} \leq \frac{1}{\binom{n-1}{p-1}} \sum_{I \ni i} \lambda_I
= \kappa_i + \frac{p-1}{n-1}(H - \kappa_i).
\]
Since $H - \kappa_i > 0$, $q + 1 \leq n - 1$, and $p \geq 2$,
\begin{equation*}
\begin{aligned}
H - q \kappa_i + (q + 1) \frac{ \binom{n-1}{p-1} }{ S_i }
\leq & H + \kappa_i + \frac{(q+1)(p-1)}{n-1} ( H - \kappa_i ) \\
\leq & p H - (p - 2) \kappa_i \leq p H.
\end{aligned}
\end{equation*}
Substitution in \eqref{eq:positive-denominator} proves the lemma.
\end{proof}

\begin{prop}[Sharp universal constant]
\label{prop:universal}
For every $\kappa\in\mathcal{P}_p$, every $i$, and every $t \in \mathbb{R}^n$ with $\sum_j S_j t_j = 0$,
\begin{equation}
Q_i \geq \frac{2 S_i}{p H} \bigg( \sum_j t_j \bigg)^2.
\label{eq:universal}
\end{equation}
Equality in \eqref{eq:universal} holds only for $t = 0$. In particular,
\[ \frac{H \mathcal{Q}_i}{S_i} > \frac{2}{p} \]
for each fixed $\kappa\notin\mathbb{R} \mathbf{1}$, while its infimum over these curvature vectors is $2/p$. The constant $2/p$ is optimal over $\mathcal{P}_p$ for every $n, p$ under consideration.
\end{prop}

\begin{proof}
The bound follows from Lemma \ref{lem:negative} and Lemma \ref{lem:positive-universal}.

To prove the sharpness, take
\[
\kappa = \Big( -(p-1)+\varepsilon, 1, \ldots, 1 \Big),\qquad 0<\varepsilon<p,
\]
and $i=1$. Then $H = q + \varepsilon$, the sums indexed by sets containing $1$ all equal $\varepsilon$, and those indexed by sets not containing $1$ all equal $p$. Let
\[
t_j = \frac{p}{q(p-\varepsilon)}, \quad j\neq 1,\qquad
t_1 = 1 - (n-1) \frac{p}{q(p-\varepsilon)}.
\]
Direct substitution gives
\[ \sum_j t_j = 1, \qquad S_1 = \frac{\binom{n-1}{p-1}}{\varepsilon}, \qquad \frac{\binom{n-1}{p}}{\binom{n-1}{p-1}}=\frac{q}{p} \]
and
\[ \begin{aligned}
\sum_j S_j t_j = & \sum_I \frac{\sum_{j \in I} t_j}{\lambda_I} \\
= & \sum_{I \ni 1} \frac{\sum_{j \in I} t_j}{\lambda_I} + \sum_{I \not\ni 1} \frac{\sum_{j \in I} t_j}{\lambda_I}  \\
= & - \frac{\binom{n-1}{p-1}}{p - \varepsilon} + \binom{n-1}{p}  \frac{p}{q(p-\varepsilon)} = 0.
\end{aligned} \]
Substitution into the two sums in \eqref{eq: Qi} gives
\[
Q_1 = \frac{\binom{n-1}{p-1}}{(p-\varepsilon)^2}
+ \binom{n-1}{p} \bigg( \frac{p}{q(p-\varepsilon)} \bigg)^2+\frac{2(n-1)\binom{n-2}{p-1}}{p\varepsilon} \bigg( \frac{p}{q(p-\varepsilon)} \bigg)^2.
\]
Using $(n-1)\binom{n-2}{p-1} = q \binom{n-1}{p-1}$ gives
\begin{equation}
\frac{H Q_1}{S_1}
= \frac{(q + \varepsilon)(n\varepsilon+2p)}{q(p-\varepsilon)^2}
\longrightarrow \frac{2}{p}.
\label{eq:sharpness}
\end{equation}
Since
\[ \frac{H \mathcal{Q}_1}{S_1} \leq \frac{H Q_1}{S_1}, \]
equation \eqref{eq:sharpness}, together with \eqref{eq:universal}, proves that $2/p$ is optimal. The ratio is invariant under positive rescaling of $\kappa$. Thus $2/p$ remains sharp after normalization to any fixed positive level of $f$.

To see the strictness,
fix $\kappa\in\mathcal{P}_p$ and suppose $\sum_j S_j t_j = 0$. Since $Q_i$ is positive definite, equality in \eqref{eq:universal} with $\sum_j t_j = 0$ implies $t = 0$. Suppose $\sum_j t_j \neq 0$.  If $\kappa_i \leq 0$, \eqref{eq:Si} gives the strict inequality
\[ S_i (H - q \kappa_i) - N + p \kappa_i \sum_{J\not\ni i} \lambda_J^{-1}
< S_i (H - q \kappa_i). \]
Combining this strict bound with \eqref{eq:constraint-identity} and Cauchy--Schwarz gives strict inequality in \eqref{eq:universal}.

If $\kappa_i \geq 0$, equality in \eqref{eq:universal} would require equality at every step in the proof of Lemma \ref{lem:positive-universal}. In particular, $\sum_{j \in J} t_j = 0$ for every $p$-set $J$ not containing $i$, and equality would hold in Cauchy--Schwarz inequality \eqref{eq:positive-local-cs}, which requires
\[
\sum_{j \in I} t_j = \frac{\lambda_I \sum_l t_l}{d_I}, \qquad
t_j = \frac{\lambda_{I_j} \sum_l t_l}{2 d_I} \quad j \notin I.
\]
For each $j\neq i$, there is a $p$-set $I$ containing $i$ and excluding $j$. Consequently all $t_j$, $j\neq i$, are nonzero and have the sign of $\sum_l t_l$. Every $\sum_{j \in J} t_j$ with $i \notin J$ is then nonzero, a contradiction. The minimum defining $\mathcal{Q}_i$ is attained by Proposition \ref{prop:schur}. Applying the strict inequality to its minimizer gives
\[ \frac{H \mathcal{Q}_i}{S_i} > \frac{2}{p} \]
at each fixed nonumbilic curvature vector, while \eqref{eq:sharpness} shows that the infimum over the cone is $2/p$.
\end{proof}

\vspace{4mm}

\section{A direction-dependent lower bound}
\label{sec:directional}

\begin{prop}[Direction-dependent coercivity]
\label{prop:directional}
For every $\kappa\in\mathcal{P}_p$ and every $t \in \mathbb{R}^n$ with $\sum_j S_j t_j = 0$,
\begin{equation}
Q_i \geq \frac{2 S_i}{H + q |\kappa_i|} \bigg( \sum_j t_j \bigg)^2.
\label{eq:directional}
\end{equation}
Consequently, for every $\theta > 0$,
\begin{equation}
|\kappa_i| \leq \frac{\theta H}{q}
\quad\Longrightarrow\quad
Q_i \geq \frac{S_i}{H} \frac{2}{1 + \theta} \bigg( \sum_j t_j \bigg)^2.
\label{eq:central-general}
\end{equation}
Taking $\theta = 1/3$ in \eqref{eq:central-general} gives the coefficient $3/2$ whenever $|\kappa_i| \leq H/(3q)$.
\end{prop}

\begin{proof}
The case $\kappa_i \leq 0$ follows from Lemma \ref{lem:negative}. Assume $\kappa_i \geq 0$.
First, we can derive the following identity:
\begin{equation}
\begin{aligned}
S_i \sum_j t_j
& = \sum_{j \neq i} (S_i - S_j) t_j \\
& = \sum_{J \not\ni i} \sum_{j \in J}
\bigg( \frac{1}{\lambda_{(J\setminus\{j\})\cup\{i\}}} - \frac{1}{\lambda_J} \bigg) t_j \\
& = \sum_{J \not\ni i} \sum_{j \in J}
\bigg( \frac{1}{\lambda_{(J\setminus\{j\})\cup\{i\}}} - \frac{1 - \tau_J}{\lambda_J} \bigg) t_j
- \sum_{J \not\ni i} \tau_J \frac{\sum_{j \in J} t_j}{\lambda_J}.
\label{eq:dual-identity}
\end{aligned}
\end{equation}
Applying weighted Cauchy--Schwarz inequality in \eqref{eq:dual-identity} gives
\begin{equation}
\begin{aligned}
& S_i^2 \bigg( \sum_j t_j \bigg)^2 \\
\leq & \sum_{J\not\ni i} \Bigg(  \tau_J^2 + \frac{1}{2} \sum_{j \in J} \lambda_J \lambda_{(J\setminus\{j\})\cup\{i\}}
\bigg( \frac{1}{\lambda_{(J\setminus\{j\})\cup\{i\}}} - \frac{1 - \tau_J}{\lambda_J} \bigg)^2 \Bigg) \cdot \\
 & \Bigg( \sum_{J\not\ni i} \frac{\big( \sum_{j \in J} t_j \big)^2}{\lambda_J^2}
+ 2 \sum_{J\not\ni i} \sum_{j\in J} \frac{t_j^2}{\lambda_J \lambda_{(J\setminus\{j\})\cup\{i\}}} \Bigg) \\
\leq & Q_i \sum_{J\not\ni i} \Phi_J (\tau_J),
\label{eq:dual-CS}
\end{aligned}
\end{equation}
where
\begin{equation}
\begin{aligned}
& \Phi_J (\tau)
= \tau^2 + \frac{1}{2} \sum_{j \in J} \lambda_J \lambda_{(J\setminus\{j\})\cup\{i\}}
\bigg( \frac{1}{\lambda_{(J\setminus\{j\})\cup\{i\}}} - \frac{1 - \tau}{\lambda_J} \bigg)^2 \\
&= \tau^2 + \frac{1}{2} \Bigg(
\lambda_J \sum_{j\in J} \frac{1}{\lambda_{(J\setminus\{j\})\cup\{i\}}} - 2 p (1 - \tau) + \bigg( p-1 + \frac{p \kappa_i}{\lambda_J} \bigg) (1-\tau)^2 \Bigg).
\end{aligned}
\label{eq:local-cost}
\end{equation}
The second equality uses
\[ \sum_{j \in J} \frac{\lambda_{(J\setminus\{j\})\cup\{i\}}}{\lambda_J} = p - 1 + \frac{p \kappa_i}{\lambda_J} > 0. \]
For each $J$, we minimize the local coefficient $\Phi_J (\tau)$:
differentiation gives
\[
\Phi_J'(\tau) = \bigg( p + 1 + \frac{p \kappa_i}{\lambda_J} \bigg) \tau + 1 - \frac{p \kappa_i}{\lambda_J}, \qquad
\Phi_J''(\tau) = p + 1 + \frac{p \kappa_i}{\lambda_J} > 0.
\]
Thus the choice
\begin{equation}
\tau_J = \frac{\frac{p \kappa_i}{\lambda_J} - 1}{p + 1 + \frac{p \kappa_i}{\lambda_J}}
\label{eq:optimal-t}
\end{equation}
minimizes $\Phi_J$. Completing the square in \eqref{eq:local-cost} gives
\begin{equation}
\begin{aligned}
\Phi_J (\tau_J)
& = \frac{1}{2} \Bigg( \lambda_J \sum_{j\in J} \frac{1}{\lambda_{(J\setminus\{j\})\cup\{i\}}} - p - 1 + \frac{p \kappa_i}{\lambda_J}
- \frac{\big( \frac{p \kappa_i}{\lambda_J} - 1 \big)^2}{p + 1 + \frac{p \kappa_i}{\lambda_J}} \Bigg) \\
& = \frac{1}{2} \Bigg( \lambda_J \sum_{j\in J} \frac{1}{\lambda_{(J\setminus\{j\})\cup\{i\}}} + 2 -
\frac{(p + 2)^2}{p + 1 + \frac{p \kappa_i}{\lambda_J}} \Bigg) \\
& \leq \frac{1}{2} \bigg( (\lambda_J + 2 \kappa_i) \sum_{j\in J} \frac{1}{\lambda_{(J\setminus\{j\})\cup\{i\}}} - p \bigg).
\label{eq:dual-cost}
\end{aligned}
\end{equation}
The last inequality is true in view of
\begin{equation*}
\begin{aligned}
& 2 - \frac{(p+2)^2}{p+1+z} \leq \frac{2 p z}{p - 1 + z} - p \qquad \text{ for } z \geq 0,
\end{aligned}
\end{equation*}
and the harmonic--arithmetic mean inequality
\[ \frac{p \frac{p \kappa_i}{\lambda_J}}{p - 1 + \frac{p \kappa_i}{\lambda_J}} \leq
\kappa_i \sum_{j\in J} \frac{1}{\lambda_{(J\setminus\{j\})\cup\{i\}}}.
\]

To estimate $\sum_{J\not\ni i} \Phi_J (\tau_J)$, we note that
\begin{gather*}
\sum_{J \not\ni i} \sum_{j\in J} \frac{1}{\lambda_{(J\setminus\{j\})\cup\{i\}}} = q S_i, \qquad p \binom{n-1}{p} = q \binom{n-1}{p-1},
\end{gather*}
and
\begin{equation*}
\begin{aligned}
& \sum_{J \not\ni i} \lambda_J \sum_{j\in J} \frac{1}{\lambda_{(J\setminus\{j\})\cup\{i\}}}
= \sum_{I \ni i} \frac{1}{\lambda_I} \sum_{j \not\in I} \lambda_{{(I\setminus\{i\})\cup\{j\}}} \\
= & \sum_{I \ni i} \frac{H - q \kappa_i + (q - 1) \lambda_I}{\lambda_I} = S_i (H - q \kappa_i) + (q - 1) \binom{n-1}{p-1}.
\end{aligned}
\end{equation*}
Thus \eqref{eq:dual-cost} implies
\begin{equation}
\sum_{J\not\ni i} \Phi_J (\tau_J) \leq \frac{1}{2} \bigg( S_i (H + q \kappa_i) - \binom{n-1}{p-1} \bigg)
\leq \frac{1}{2} S_i (H + q \kappa_i).
\label{eq:total-dual-cost}
\end{equation}
Combining this bound with \eqref{eq:dual-CS} proves \eqref{eq:directional}. Substituting the condition $|\kappa_i| \leq \theta H/q$ gives \eqref{eq:central-general}.

For strictness and degenerate cases, we note that
the direction-dependent inequality \eqref{eq:directional} is also strict for nonzero $t$ satisfying $\sum_j S_j t_j = 0$. For $\kappa_i \leq 0$ this follows from the preceding strict inequality in \eqref{eq:Si}.
For $\kappa_i \geq 0$ and $\sum_j t_j \neq 0$, \eqref{eq:dual-CS} and \eqref{eq:total-dual-cost} give
\[
S_i^2 \bigg( \sum_j t_j \bigg)^2 \leq Q_i \sum_{J \not\ni i} \Phi_J (\tau_J)
\leq \frac{1}{2} \bigg( S_i (H + q \kappa_i) - \binom{n-1}{p-1} \bigg) Q_i
< \frac{1}{2} S_i (H + q \kappa_i) Q_i.
\]
If $\sum_j t_j = 0$ and $t \neq 0$, the strict inequality follows from the positive definiteness of $Q_i$.
\end{proof}

Combining Proposition \ref{prop:universal} and Proposition \ref{prop:directional} gives
\begin{equation}
Q_i \geq \frac{S_i}{H} \max\bigg\{\frac{2}{p},\frac{2H}{H+q|\kappa_i|}\bigg\} \bigg( \sum_j t_j \bigg)^2
\qquad\text{whenever } \sum_j S_j t_j = 0
\label{eq:combined-coercivity}
\end{equation}
at every admissible curvature vector, including those with all components equal.
The coefficients give explicit lower bounds for the optimal coefficients $\frac{H \mathcal{Q}_i}{S_i}$ wherever it is defined; equality is not asserted.

\vspace{4mm}

\section{Uniform global curvature estimate}
\label{sec:first-absorption}

We continue our proof of Theorem \ref{thm:main} from \eqref{eq: lead-inequality}, which can be rewritten as
\begin{equation} \label{eq:lead-inequality}
\begin{aligned}
& (\beta - 1) \bigg( \sum_i \frac{S_i}{S} + \sum_i \frac{S_i}{S} \kappa_i^2 \bigg)  +  \frac{\sum_i Q_i}{H S}
\\ &
+ 2 \beta \sum_i \frac{S_i}{S} \frac{u_i^2}{u^2} \frac{\kappa_i - \nu^{n+1}}{\nu^{n + 1}} - \beta (\beta - 1) \sum_i \frac{S_i}{S} \frac{u_i^2}{u^2} \bigg( \frac{\kappa_i - \nu^{n + 1}}{\nu^{n + 1}} \bigg)^2 < 0.
\end{aligned}
\end{equation}

First, we give an estimate for
$\sum \frac{S_i}{S} \kappa_i^2$.
By \eqref{eq:universal} and \eqref{eq4-6}, we have
\begin{equation} \label{Fact0}
\frac{Q_i}{H S} \geq \frac{2}{p} \frac{S_i}{S} \bigg( \frac{H_i}{H} \bigg)^2 = \frac{2 \beta^2}{p} \frac{S_i}{S} \frac{u_i^2}{u^2} \bigg( \frac{\kappa_i - \nu^{n + 1}}{\nu^{n + 1}} \bigg)^2.
\end{equation}
Moreover, Cauchy--Schwartz gives
\[ \begin{aligned}
\bigg| \sum \frac{S_i}{S} \frac{u_i^2}{u^2} \frac{\kappa_i - \nu^{n+1}}{\nu^{n + 1}} \bigg|
\leq & \sqrt{\sum_i \frac{S_i}{S} \frac{u_i^2}{u^2}} \sqrt{\sum_i \frac{S_i}{S} \frac{u_i^2}{u^2} \bigg( \frac{\kappa_i - \nu^{n + 1}}{\nu^{n + 1}} \bigg)^2} \\
\leq & \sqrt{\sum_i \frac{S_i}{S} \frac{u_i^2}{u^2} \bigg( \frac{\kappa_i - \nu^{n + 1}}{\nu^{n + 1}} \bigg)^2}.
\end{aligned}
\]
The last inequality is true in view of \eqref{eq:weight-sums} and \eqref{eq2-2}:
\begin{equation} \label{eq2-2-2}
\sum_i \frac{u_i^2}{u^2} = 1 - (\nu^{n + 1})^2.
\end{equation}
Hence \eqref{eq:lead-inequality} reduces to
\begin{equation} \label{eq: estimate-whole}
\begin{aligned}
& (\beta - 1) \bigg( p + \sum \frac{S_i}{S} \kappa_i^2 \bigg) + \bigg( \frac{2 \beta^2}{p} - \beta (\beta - 1) \bigg) \sum_i \frac{S_i}{S} \frac{u_i^2}{u^2} \bigg( \frac{\kappa_i - \nu^{n + 1}}{\nu^{n + 1}} \bigg)^2
\\ &
- 2 \beta \sqrt{\sum_i \frac{S_i}{S} \frac{u_i^2}{u^2} \bigg( \frac{\kappa_i - \nu^{n + 1}}{\nu^{n + 1}} \bigg)^2} < 0.
\end{aligned}
\end{equation}
Requiring
\begin{equation} \label{Condition1}
 \frac{2 \beta^2}{p} - \beta (\beta - 1) > 0,
\end{equation}
and completing the square in \eqref{eq: estimate-whole} we obtain
\begin{equation} \label{rough-estimate}
\sum \frac{S_i}{S} \kappa_i^2
< \frac{4 \beta^2}{4 (\beta - 1) \bigg( \frac{2 \beta^2}{p} - \beta (\beta - 1) \bigg)} -  p.
\end{equation}
We further require that
\begin{equation} \label{Condition2}
 \frac{\beta^2}{(\beta - 1) \bigg( \frac{2 \beta^2}{p} - \beta (\beta - 1) \bigg)} >  p.
\end{equation}
Choosing
\[ \beta = \frac{p + 1}{p} > 1, \]
which satisfies both \eqref{Condition1} and \eqref{Condition2}, the estimate \eqref{rough-estimate} becomes
\begin{equation} \label{Rough-estimate}
\sum \frac{S_i}{S} \kappa_i^2
< \frac{p (p^2 - 2)}{p + 2} < p^2.
\end{equation}

By Lemma \ref{lem:elementary}, $\frac{S_n}{S} \geq \frac{p}{n}$. Togetherw ith \eqref{Rough-estimate}, this yields
\[
\frac{p}{n} \kappa_n^2 \leq \sum_i \frac{S_i}{S} \kappa_i^2 < p^2,
\]
so for any $i$,
\begin{equation} \label{lower-bound}
\kappa_i \geq \kappa_n >  -\sqrt{np}.
\end{equation}

Next, we come back to \eqref{eq:lead-inequality} to derive the uniform global curvature estimate.
For the index $i$ such that $|\kappa_i| \geq \frac{H}{3q}$,
suppose $\frac{H}{3 q} \geq \sqrt{np}$ (otherwise we are done).
Then
\begin{equation} \label{Fact1}
\kappa_i \geq \frac{H}{3q} \geq \sqrt{np} > 1 \geq \nu^{n + 1}.
\end{equation}
Meanwhile, by \eqref{Fact0} we have
\begin{equation} \label{Fact2}
\begin{aligned}
& \frac{Q_i}{H S} - \beta (\beta - 1) \frac{S_i}{S} \frac{u_i^2}{u^2} \bigg( \frac{\kappa_i - \nu^{n + 1}}{\nu^{n + 1}} \bigg)^2 \\
\geq & \bigg( \frac{2 \beta^2}{p} - \beta (\beta - 1) \bigg) \frac{S_i}{S} \frac{u_i^2}{u^2} \bigg( \frac{\kappa_i - \nu^{n + 1}}{\nu^{n + 1}} \bigg)^2 \geq 0.
\end{aligned}
\end{equation}

For an index $i$ with $|\kappa_i| \leq \frac{H}{3q}$,  Proposition \ref{prop:directional} gives
\begin{equation*}
\begin{aligned}
& \frac{Q_i}{H S} - \beta (\beta - 1) \frac{S_i}{S} \frac{u_i^2}{u^2} \bigg( \frac{\kappa_i - \nu^{n + 1}}{\nu^{n + 1}} \bigg)^2
\geq  \bigg( \frac{3 \beta^2}{2} - \beta (\beta - 1) \bigg) \frac{S_i}{S} \frac{u_i^2}{u^2} \bigg( \frac{\kappa_i - \nu^{n + 1}}{\nu^{n + 1}} \bigg)^2.
\end{aligned}
\end{equation*}
Then completing the square in $\frac{\kappa_i - \nu^{n + 1}}{\nu^{n + 1}}$ yields
\begin{equation}
2 \beta \frac{S_i}{S} \frac{u_i^2}{u^2} \frac{\kappa_i - \nu^{n+1}}{\nu^{n + 1}} + \frac{Q_i}{H S} - \beta (\beta - 1) \frac{S_i}{S} \frac{u_i^2}{u^2} \bigg( \frac{\kappa_i - \nu^{n + 1}}{\nu^{n + 1}} \bigg)^2
\geq -\frac{2(p+1)}{3p+1} \frac{S_i}{S} \frac{u_i^2}{u^2}.
\label{eq:central-square}
\end{equation}

Let $\mathcal{C}$ be the set of indices $i$ such that $|\kappa_i| \leq \frac{H}{3q}$. We prove the nonnegativity of the sum of the $i \in \mathcal{C}$ layers along with the terms
\[ (\beta - 1) \frac{S_i}{S} \]
with $i\notin\mathcal{C}$. Applying \eqref{eq:central-square} to this sum we obtain
\begin{equation} \label{Fact3}
\begin{aligned}
&  (\beta - 1) \sum_i \frac{S_i}{S} +  \sum_{i \in \mathcal{C}} \frac{Q_i}{H S} \\
& + 2 \beta \sum_{i \in \mathcal{C}} \frac{S_i}{S} \frac{u_i^2}{u^2} \frac{\kappa_i - \nu^{n+1}}{\nu^{n + 1}} - \beta (\beta - 1) \sum_{i \in \mathcal{C}} \frac{S_i}{S} \frac{u_i^2}{u^2} \bigg( \frac{\kappa_i - \nu^{n + 1}}{\nu^{n + 1}} \bigg)^2 \\
\geq & 1 - \sum_{i \in \mathcal{C}} \frac{2(p+1)}{3p+1} \frac{S_i}{S} \frac{u_i^2}{u^2}
\geq  1 - \frac{2(p+1)}{3p+1} = \frac{p - 1}{3 p + 1} > 0.
\end{aligned}
\end{equation}
The last inequality is true by \eqref{eq:weight-sums} and \eqref{eq2-2-2}.
Now we apply \eqref{Fact1}, \eqref{Fact2} and \eqref{Fact3} to \eqref{eq:lead-inequality} to reach a contradiction. We thus finish the proof of Theorem \ref{thm:main}.

\begin{proof}[Proof of Theorem~\ref{thm:plateau}]
Lemma~\ref{lem:structure}, in particular \eqref{eq:ray-growth}, verifies the structural hypotheses of the Guan--Spruck approximation theorem \cite[Theorem~1.3]{GS10}. For every sufficiently small $\epsilon>0$,
that theorem supplies a smooth admissible solution $u^\epsilon$
of \eqref{eq1-4}, together with its height and angle estimates.
By the boundary second-derivative estimate \cite[Theorem~5.1]{GS10} and the graph formula \eqref{eq:graph-matrix}, the hyperbolic principal curvatures on the finite boundary are bounded independently of $\epsilon$. Theorem~\ref{thm:main}
therefore gives the uniform bound on the whole graph. Positive height lower bounds on compact subsets of $\Omega$ allow the compactness argument of \cite[Section~4]{GS10} to produce a smooth admissible limit as $\epsilon\downarrow0$. The boundary angle and $C^1$ regularity of $u$, together with the $C^{1,1}$ regularity of $u^2$, follow from the barrier and boundary estimates in \cite[Sections~3--5]{GS10}. The admissible comparison argument there gives uniqueness.
The global bound for $u|D^2u|$ follows from
\eqref{eq:graph-matrix}, the curvature bound and the angle estimate.
The graph is complete because its induced metric is
$u^{-2}(I_n+Du\otimes Du)$ and $u$ vanishes continuously on
$\partial\Omega$. The same boundary behavior makes the embedding
proper in $\mathbb H^{n+1}$ and identifies its asymptotic boundary
with $\partial\Omega\times\{0\}$.
\end{proof}

\medskip
\noindent\textbf{Acknowledgements.}\quad
The first author was supported by the Natural Science Foundation of Heilongjiang Province (Grant No. PL2025A006) and Program for Young Talents of Basic Research in University of Heilongjiang Province (Grant No. YQJH2025116). The second author was supported by National Natural Science Foundation of China (No. 12571212). The authors acknowledge the use of AI tools. All mathematical statements and proofs were independently verified by the authors, who take full responsibility for the content of the manuscript.

\vspace{4mm}
\noindent\textbf{Conflict of interest.}\quad The authors declare that there is no conflict of interest.

\vspace{4mm}
\noindent\textbf{Data availability.}\quad Data sharing is not applicable to this article because no datasets were generated or analyzed during the current study.

\end{document}